\documentclass[a4paper,11pt]{article}

\usepackage{amscd}
\usepackage{amssymb,amsfonts,amsmath,amsthm}
\usepackage{verbatim}

\usepackage{amsmath}
\usepackage{amsthm}
\usepackage{amssymb}

\usepackage{latexsym}
\usepackage{array}
\usepackage{enumerate}

\numberwithin{equation}{section}

\usepackage{amsmath,amssymb,mathrsfs,amsthm}
\usepackage[all]{xy}
\usepackage{graphicx}
\usepackage{ulem}
\usepackage{ascmac}
\usepackage{mathtools}

\newcommand{\Mod}{\mathrm{Mod}}
\newcommand{\Hom}{\mathrm{Hom}}

\newcommand{\ihom}{\mathcal{I}hom}

\newcommand{\M}{\mathcal{M}}
\newcommand{\N}{\mathcal{N}}

\newcommand{\Ker}{\operatorname{Ker}}
\newcommand{\Coker}{\operatorname{Coker}}
\newcommand{\Image}{\operatorname{Im}}

\newcommand{\id}{{\rm id}}

\newcommand{\simto}{\overset{\sim}{\longrightarrow}}

\newcommand{\op}{{\mbox{\scriptsize op}}}

\newcommand{\SA}{\mathcal{A}}
\newcommand{\SM}{\mathcal{M}}

\newcommand{\SF}{\mathcal{F}}

\newcommand{\Modcoh}{\mathrm{Mod}_{\mbox{\rm \scriptsize coh}}}

\renewcommand{\hom}{\mathcal{H}om}

\newcommand{\I}{{\rm I}}
\newcommand{\Icoh}{{\rm I}_{\scriptsize \rm coh}}

\newtheorem*{maintheorem}{Main Theorem}
\newtheorem{theorem}{Theorem}[section]

\newtheorem{corollary}[theorem]{Corollary}
\newtheorem{lemma}[theorem]{Lemma}
\newtheorem{sublemma}[theorem]{Sublemma}
\newtheorem{proposition}[theorem]{Proposition}

\theoremstyle{definition}
\newtheorem{definition}[theorem]{Definition}
\theoremstyle{remark}

\title{Finiteness properties of Ind-Sheaves with Ring Actions}

\author{Yohei ITO\footnote{Department of Mathematics, Faculty of Science Division II, Tokyo University of Science, 1-3, Kagurazaka, Shinjuku-ku, Tokyo, 162-8601, Japan. E-mail: yitoh@rs.tus.ac.jp }}

\date{}

\begin{document}
\maketitle

\begin{abstract}
In this paper, we shall consider some finiteness of ind-sheaves with ring actions.
As the main result of this paper, there exists an equivalence of categories between
the abelian category of coherent ind-$\beta_X\SA$-modules and the one of coherent $\SA$-modules,
where $\SA$ is a sheaf of $\Bbbk_X$-algebras and $\Bbbk$ is a field.
\end{abstract}

\section{Introduction}
In \cite{KS01}, M.\:Kashiwara and P.\:Schapira introduced the notion of ind-sheaves
to treat ``sheaves" of functions with tempered growth conditions.
Ind-sheaves are defined as ind-objects of the category of sheaves of vector spaces with compact support as below.

Throughout this paper, 
$\Bbbk$ is a field and all topological spaces are locally compact, Hausdorff and
with a countable basis of open sets.
We denote by $\Mod(\Bbbk_X)$ the category of sheaves 
of $\Bbbk$-vector spaces on a topological space $X$,
and denote by $\Mod^c(\Bbbk_X)$ the full subcategory of $\Mod(\Bbbk_X)$ 
consisting of sheaves with compact support.

\begin{definition}[{\cite[Def.\:4.1.2]{KS01}}]
An ind-sheaf of $\Bbbk$-vector spaces on $X$ is an ind-object of $\Mod^c(\Bbbk_X)$,
that is, inductive limit 
$$\displaystyle``\varinjlim_{i\in I}"\SF_i := \varinjlim_{i\in I}\Hom_{\Mod^c(\Bbbk_X)}(\ \cdot\ ,\ \SF_i)$$
 of a small filtrant inductive system $\{\SF_i\}_{i\in I}$ in $\Mod^c(\Bbbk_X)$.
 \end{definition}
 
  We call an ind-sheaf of $\Bbbk$-vector spaces an ind-sheaf, for simplicity.
 Let us denote by $\I\Bbbk_X$ the category of ind-sheaves of $\Bbbk$-vector spaces on $X$.
 Note that it is abelian (see \cite[Thm.\:8.6.5 (1)]{KS06} for the details).
 Note also that there exists a natural exact embedding
$\iota_X\colon \Mod(\Bbbk_X)\to\I\Bbbk_X$
of abelian categories.
It has an exact left adjoint $\alpha_X\colon \I\Bbbk_X\to\Mod(\Bbbk_X)$, 
that has in turn an exact fully faithful
left adjoint functor $\beta_X\colon \Mod(\Bbbk_X)\to\I\Bbbk_X$
(see \cite[\S 3.3]{KS01} for the details): 
 \[\xymatrix@C=60pt{\Mod(\Bbbk_X)  \ar@<1.0ex>[r]^-{\iota_{X}} 
 \ar@<-1.0ex>[r]_- {\beta_{X}} & \I\Bbbk_X 
\ar@<0.0ex>[l]|-{\alpha_{X}}}.\]
For a continuous map $f\colon X\to Y$,
we have the tensor product functor
$\otimes\colon \I\Bbbk_X\times \I\Bbbk_X\to \I\Bbbk_X$,
the internal hom functor 
$\ihom\colon(\I\Bbbk_X)^\op\times \I\Bbbk_X\to \I\Bbbk_X$,
the external hom functor 
$\hom := \alpha_X\ihom\colon(\I\Bbbk_X)^\op\times \I\Bbbk_X\to \Mod(\Bbbk_X)$,
the inverse image functor $f^{-1}\colon \I\Bbbk_Y\to \I\Bbbk_X$
and the direct image functor $f_{\ast}\colon \I\Bbbk_X\to \I\Bbbk_Y$.
For $Z\subset X$ and the inclusion map $i_Z\colon Z\to X$,
we denote $i_Z^{-1}(\cdot)$ by $(\cdot)|_Z$, for simplicity.
Note that these functors have same properties of the case of classical sheaves.
See \cite[\S\S\:4.2, 4.3]{KS01} for the details.
%We just remark that for $F, G\in \I\Bbbk_X$ a sheaf $\hom(F, G)$ is a sheaf which is defined by
%$$\Gamma(U; \hom(F, G)) := \Hom_{\I\Bbbk_U}(F|_U, G|_U)$$
%for an open subset $U$ of $X$
%by \cite[Props.\:3.3.11 (i), 5.4.11]{KS01}.

The authors also defined the notion of stalk for ind-sheaves which is called ind-stalk,
and prove the following results.

\begin{definition}[{\cite[\S\:4.3]{KS01}}]
Let $X$ be a topological space.
For a locally closed subset $Z\subset X$ and an ind-sheaf $F\in \I\Bbbk_X$,
we set
$$_ZF := F\otimes \beta_X\Bbbk_Z.$$
Moreover, we define the ind-stalk of an ind-sheaf $F\in \I\Bbbk_X$ at $x\in X$ as $_{\{x\}}F$.
For simplicity, we denote $_{\{x\}}F$ by $_xF$.
\end{definition}

Remark that there exists an isomorphism $_ZF \simeq i_{Z!!}i_Z^{-1}F$,
here $i_Z\colon Z\to X$ is the inclusion map and
$i_{Z!!}\colon \I\Bbbk_Z\to \I\Bbbk_X$ is the left adjoint functor of $i_Z^{-1}\colon \I\Bbbk_X\to \I\Bbbk_Z$.
See \cite[Notation\:2.3.5, Def.\:3.3.15, Prop.\:4.2.14 (i)]{KS01} for the details.

\begin{proposition}[{\cite[Prop.\:4.3.21]{KS01}}]\label{prop_indstalk}
Let $\varphi\colon F\to G$ be a morphism of ind-sheaves.
Assume that for any $x\in X$, the induced morphism $_x\varphi\colon _xF\to _xG$ is the zero morphism.
Then $\varphi$ is the zero morphism.
\end{proposition}

We use the following corollary in the last step of the proof of Main Theorem (Theorem \ref{main-thm}).
\begin{corollary}\label{cor-indstalk}
Let $\varphi\colon F\to G$ be a morphism of ind-sheaves.
Assume that for any $x\in X$, there exists an open neighborhood $U$ of $x$ such that
the induced morphism ${}_U\varphi\colon {}_UF\to {}_UG$ is the zero morphism.
Then $\varphi$ is the zero morphism.
\end{corollary}

\begin{proof}
By the assumption, for any $x\in X$, there exists an open neighborhood $U$ of $x$ such that
the induced morphism ${}_U\varphi\colon {}_UF\to {}_UG$ is the zero morphism.
Then the induced morphism ${}_x({}_U\varphi)\colon {}_x({}_UF)\to {}_x({}_UG)$ is the zero morphism.
Since there exist isomorphisms
$${}_x({}_UF) = {}_UF\otimes \beta_X\Bbbk_{\{x\}}
= (F\otimes \beta_X\Bbbk_U)\otimes\beta_X\Bbbk_{\{x\}}
\simeq F\otimes \beta_X(\Bbbk_U\otimes \Bbbk_{\{x\}})
\simeq F\otimes \beta_X(\Bbbk_{\{x\}})
\simeq {}_xF$$
and also ${}_x({}_UG) \simeq {}_xG$,
the induced morphism ${}_x({}_U\varphi)$ 
is equal to the induced morphism ${}_x\varphi\colon {}_xF\to {}_xG$.
Hence $\varphi$ is the zero morphism by Proposition \ref{prop_indstalk}.
\end{proof}

Moreover, the authors defined the notion of ind-modules as below. 
\begin{definition}[{\cite[\S\S\:5.4, 5.5]{KS01}}]
Let $X$ be a topological space.
\begin{itemize}
\item[(1)]
An ind-$\Bbbk_X$-algebra is the date of an ind-sheaf $A$ and morphisms of ind-sheaves
$$\mu_A\colon A\otimes A\to A,\ \varepsilon_A\colon \iota_X\Bbbk_X\to A$$
such that the following diagrams are commutative:
\[\hspace{-27pt}
\xymatrix@M=5pt@R=20pt@C=40pt{
A\ar@{->}[r]^-\sim\ar@{->}[d]_-{\rm id_{A}} & \iota_X\Bbbk_{X}\otimes A\ar@{->}[d]^-{\varepsilon_{A}\otimes A}\\
A & A\otimes A\ar@{->}[l]^-{\mu_{A}},
}\hspace{11pt}
\xymatrix@M=5pt@R=20pt@C=40pt{
A\ar@{->}[r]^-\sim\ar@{->}[d]_-{\rm id_{A}} & A\otimes\iota_X\Bbbk_{X}\ar@{->}[d]^-{A\otimes\varepsilon_{A}}\\
A &  A\otimes A\ar@{->}[l]^-{\mu_{A}},
}\hspace{11pt}
\xymatrix@M=5pt@R=20pt@C=40pt{
A\otimes A\otimes A\ar@{->}[r]^-{\mu_A\otimes A}\ar@{->}[d]_-{A\otimes\mu_A} &
A\otimes A\ar@{->}[d]^-{\mu_{A}}\\
A\otimes A \ar@{->}[r]_-{\mu_{A}} & A.
}\]

\item[(2)]
Let $A$ be an ind-$\Bbbk_X$-algebra.
An ind-$A$-module (or $A$-module, for simplicity) is the date of an ind-sheaf $M$
and a morphism $\mu_M\colon A\otimes M\to M$
 such that the following diagrams are commutative:
 \[\hspace{-27pt}
\xymatrix@M=5pt@R=20pt@C=40pt{
M\ar@{->}[r]^-\sim\ar@{->}[d]_-{\rm id_{M}} & \iota_X\Bbbk_{X}\otimes M\ar@{->}[d]^-{\varepsilon_{A}\otimes M}\\
M & A\otimes M\ar@{->}[l]^-{\mu_{M}},
}\hspace{21pt}
\xymatrix@M=5pt@R=20pt@C=40pt{
A\otimes A\otimes M\ar@{->}[r]^-{\mu_A\otimes M}\ar@{->}[d]_-{A\otimes\mu_M} &
A\otimes M\ar@{->}[d]^-{\mu_{M}}\\
A\otimes M \ar@{->}[r]_-{\mu_{M}} & M.
}\]

\item[(3)]
A morphism of ind-$A$-modules from $M$ to $N$ is a morphism $\varphi\colon M\to N$ of ind-sheaves
such that the following diagram is commutative:
 \[\hspace{-27pt}
\xymatrix@M=5pt@R=20pt@C=40pt{
A\otimes M\ar@{->}[r]^-{A\otimes \varphi}\ar@{->}[d]_-{\mu_M}
& A\otimes N\ar@{->}[d]^-{\mu_N}\\
M\ar@{->}[r]_-{\varphi} & N.
}\]
\end{itemize}
\end{definition}

%\begin{remark}
%Let $A$ be an ind-$\Bbbk_X$-algebra.
%For an ind-$A$-module $M$,
%we denote by $\nu_M\colon M\to\ihom(A, M)$ the corresponding morphism of $\mu_M$
%by the isomorphism $\Hom_{\I\Bbbk_X}(A\otimes M, M)\simeq \Hom_{\I\Bbbk_X}(M, \ihom(A, M))$
%and set
%\begin{align*}
%e_M& \colon M\simeq \Bbbk_X\otimes M\xrightarrow{\ \varepsilon_A\otimes M} A\otimes M,\\
%e_M^\ast&\colon \ihom(A, M)\xrightarrow{\ \ihom(\varepsilon_A, M)\ }\ihom(\Bbbk_X, M)\simeq M.
%\end{align*}
%Note that in the classical case of module over ring,
%the analogous morphisms of $\mu_M, \nu_M, e_M$ and $e_M^\ast$ are the morphisms
%$a\otimes m\mapsto am, m\mapsto (a\mapsto am), m\mapsto 1\otimes m$
%and $\varphi\mapsto \varphi(1)$, respectively.
%Moreover, for an ind-$\Bbbk_M$-algebra $R$
%we defne an ind-$\Bbbk_M$-algebra $R^\op$
%by the date of an ind-$\Bbbk_M$-algebra $R$, a morphism
%$\mu_{R^\op}\colon R\otimes R \to R$ of ind-sheaves corresponding to $a\otimes b\mapsto b\otimes a$
%and a morphism $\varepsilon_{A^\op} := \varepsilon_{A}$.
%\end{remark}

Let us denote by $\I A$ the category of ind-$A$-modules.
Note that this is abelian, see \cite[\S\:5.4]{KS01} for the details.
Note also that
we have bifunctors:
$$\otimes_A\colon \I(A^\op)\times \I A\to \I\Bbbk_X,
\hspace{7pt}
%\mbox{and} 
%\hspace{7pt}
\ihom_A\colon(\I A)^\op\times \I A\to \I\Bbbk_X,
\hspace{7pt}
\hom_A := \alpha_X\ihom_A\colon(\I A)^\op\times \I A\to \Mod(\alpha_XA).$$
Moreover, for a continuous map $f\colon X\to Y$ and ind-$\Bbbk_Y$-algebra $B$,
we have the inverse image functor and  the direct image functor:
$$f^{-1}\colon \I B\to \I(f^{-1}B),
\hspace{17pt}
f_{\ast}\colon \I(f^{-1}B)\to \I B.$$
Note that these functors have same properties of the case of sheaves with ring actions.
See \cite[\S\S\:4.2, 4.3]{KS01} for the details. 
Here, we just recall the following properties.

\begin{proposition}[{\cite[Ex.\:5.4.7 for (1), (2), p92.\:for (3), Props.\:3.3.11 (i), 5.4.11]{KS01}}]
Let $X$ be a topological space, $A$ an ind-$\Bbbk_X$-algebra and $\SA$ a sheaf of $\Bbbk_X$-algebras on $X$.
\begin{itemize}
\item[\rm (1)]
Ind-sheaves $\iota_X\SA$ and $\beta_X\SA$ are ind-$\Bbbk_X$-algebras.

\item[\rm (2)]
The functors $\iota_X$ and $\beta_X$ induce fully faithful functors
$\iota_X\colon \Mod(\SA)\to\I(\iota_X\SA)$ and 
$\beta_X\colon \Mod(\SA)\to\I(\beta_X\SA)$,
and the functor $\alpha_X$ induces functor
$\alpha_X\colon \I(\iota_X\SA)\to \Mod(\SA)$ and 
$\alpha_X\colon \I(\beta_X\SA)\to \Mod(\SA)$.

\item[\rm (3)]
For any $M\in\I A$, we have an isomorphism
$A\otimes_AM \simeq M$ and $\ihom_A(A, M)\simeq M$.

\item[\rm (4)]
For any $\M, \N\in\Mod(\SA)$, 
we have an isomorphism
$\beta_X(\M\otimes_{\SA}\N)\simeq \beta_X\M\otimes_{\beta_X\SA}\beta_X\N$.

\item[\rm (5)]
For any $M, N\in\I A$, we have isomorphisms
$\alpha_X(M\otimes_{\beta_X\SA} N)\simeq \alpha_XM\otimes_\SA\alpha_XN$.

\item[\rm (6)]
Let $f\colon X\to Y$ be a continuous map and $B$ an ind-$\Bbbk_Y$-algebras.
Then for any $M, N\in \I B$, we have an isomorphism
$f^{-1}(M\otimes_BN)\simeq f^{-1}M\otimes_{f^{-1}B}f^{-1}N$.

\item[\rm (7)]
For $M, N\in \I A$ and an open subset $U$ of $X$,
then we have
$$\hom_A(M, N) = \Hom_{\I(A|_U)}(M|_U, N|_U).$$
\end{itemize}
\end{proposition}

For simplicity, let ind-$\beta_X\SA$-module be called ind-$\SA$-module.
Note that ind-$\SA$-modules are not necessarily ind-objects of $\Mod^c(\SA)$.

The fully faithful functor $\beta_X\colon \Mod(\SA)\to \I(\beta_X\SA)$
is not necessarily an equivalence of categories.
In this paper, we define the notion of coherent ind-modules in Definition \ref{def-coh}
and prove the following results.
%that 
%the functor $\beta_X\colon \Mod(\SA)\to \I(\beta_X\SA)$
%induces an equivalence of abelian categories
%between the abelian category $\Modcoh(\SA)$ of coherent $\SA$-modules
%and the one $\Icoh(\beta_X\SA)$ of coherent ind-$\SA$-modules
%in Theorem \ref{main-thm}.
Although it may be known to experts,
it is not in the literature to our knowledge.

\begin{maintheorem}[Theorem \ref{main-thm}]
Let $\SA$ be a sheaf of $\Bbbk_X$-algebras.
There exists an equivalence of categories:
 \[\xymatrix@C=60pt{\Modcoh(\SA)  \ar@<1.0ex>[r]^-{\beta_{X}} \ar@{}[r]|-\sim
 & \Icoh(\beta_X\SA) 
\ar@<1.0ex>[l]^-{\alpha_{X}}}.\]
\end{maintheorem}

At the end of introduction,
let us recall the definition of coherent modules.
\begin{definition}
Let $X$ be a topological space and $\SA$ a sheaf of $\Bbbk_X$-algebras.%\footnote{One can also consider more general settings.}.
\begin{itemize}
\item[\rm (1)]
An $\SA$-module $\M$ is called locally finitely generated over $\SA$
if for any $x\in X$ there exist a positive integer $n$, an open neighborhood $U$ of $x$
and an exact sequence $(\SA|_U)^{\oplus n}\to \SM|_U\to 0$.

\item[\rm (2)]
An $\SA$-module $\M$ is called locally finitely presented over $\SA$
if for any $x\in X$ there exist two positive integers $m, n$, an open neighborhood $U$ of $x$
and an exact sequence $(\SA|_U)^{\oplus m}\to (\SA|_U)^{\oplus n}\to \SM|_U\to 0$.

\item[\rm (3)]
An $\SA$-module $\M$ is called pseudo-coherent over $\SA$
if for any open subset $U$ of $X$, any locally finitely generated $\SA|_U$-submodule $\N$ of $\M|_U$
is locally finitely presented over $\SA|_U$.

\item[\rm (4)]
An $\SA$-module $\M$ is called coherent over $\SA$
if it is locally finitely generated over $\SA$ and pseudo-coherent over $\SA$.
\end{itemize}
\end{definition}

\section*{Acknowledgement}
I would like to thank Taito Tauchi and Masatoshi Kitagawa
for discussions and giving comments.
%
%This work was supported by Grant-in-Aid for Research Activity Start-up (No.\:21K20335), 
%Japan Society for the Promotion of Science.

\section{Main Results}
The main theorem of this paper is Theorem \ref{main-thm}.
Recall that throughout of this paper,
$\Bbbk$ is a field and all topological spaces are locally compact, Hausdorff and
with a countable basis of open sets.
Remark that the category $\I\Bbbk_X$ of ind-sheaves admits finite coproducts
which are commute with functors $\iota_X, \alpha_X, \beta_X$.
Let us denote by $\oplus$ the finite coproduct in $\I\Bbbk_X$. 

\begin{definition}
Let $A$ be an ind-$\Bbbk_X$-algebra and $M$ an ind-$A$-module.
We say that $N$ is an ind-$A$-submodule of $M$
if it is a subobject of $M$ in the category $\I A$ of ind-$A$-modules.
\end{definition}

Similar to the case of sheaves with ring action,
we shall define some finiteness properties for ind-modules as below\footnote{Definition 2.2 is extracted from \cite[Exe. 8.23]{KS06}}.
\begin{definition}\label{def-coh}
Let $X$ be a topological space and $A$ an ind-$\Bbbk_X$-algebra.

\begin{itemize}
\item[\rm (1)]
We say that 
an ind-$A$-module $M$ is locally finitely generated over $A$
if for any $x\in X$ there exist a positive integer $n$, an open neighborhood $U$ of $x$
and an exact sequence $(A|_U)^{\oplus n}\to M|_U\to 0$ in $\I(A|_U),$
where $A|_U$ (resp.\,$M|_U$) is the inverse image of $A$ (resp.\,$M$)
by the natural embedding $i_U\colon U\hookrightarrow X$.

\item[\rm (2)]
We say that
an ind-$A$-module $M$ is locally finitely presented over $A$
if for any $x\in X$ there exist two positive integers $m, n$, an open neighborhood $U$ of $x$
and an exact sequence 
$(A|_U)^{\oplus m}\to (A|_U)^{\oplus n}\to M|_U\to 0$ in $\I(A|_U).$

\item[\rm (3)]
We say that
an ind-$A$-module $M$ is pseudo-coherent over $A$
if for any open subset $U$ of $X$, any locally finitely generated ind-$A|_U$-submodule $N$ of $M|_U$
is locally finitely presented over $A|_U$.

\item[\rm (4)]
We say that
an ind-$A$-module $M$ is coherent over $A$
if it is locally finitely generated and pseudo-coherent over $A$.
\end{itemize}
\end{definition}

Let $\SA$ be a sheaf of $\Bbbk_X$-algebras.
Note that for any $\SA$-module $\M$,
any epimorphism $\varphi\colon \M\twoheadrightarrow \N$ of $\SA$-modules
and any morphism $f\colon \SA\to \N$ of $\SA$-modules,
there locally exists a morphism $\widetilde{f}\colon \SA\to \M$ of $\SA$-modules
such that $\varphi\circ \widetilde{f} = f$,
namely, for any $x\in X$ there exist an open neighborhood $U$ of $x$
and a morphism $\widetilde{f|_U}\colon \SA|_U\to \M|_U$
such that $\varphi|_U\circ \widetilde{f|_U} = f|_U$.
Any ind-$\Bbbk_X$-algebra has same property.
\begin{lemma}\label{lem0}
Let $A$ be an ind-$\Bbbk_X$-algebra,
$M, N$ ind-$A$-modules,
and $\varphi\colon M\twoheadrightarrow N$ an epimorphism of ind-$A$-modules.
\begin{itemize}
\item[\rm(1)]
Let $f\colon A\to N$ be a morphism of ind-$A$-modules.
For any $x\in X$, there exist an open neighborhood $U$ of $x$ and
a morphism $\widetilde{f|_U}\colon A|_U\to M|_U$ of ind-$A|_U$-modules
such that $\varphi|_U\circ\widetilde{f|_U} = f|_U:$
\[\xymatrix@C=45pt@R=25pt{
M|_U\ar@{->>}[r]^-{\varphi|_U} & N|_U\\
A|_U. \ar@{.>}[u]^-{\widetilde{f|_U}}\ar@{->}[ru]_-{f|_U} & {}
}\]

\item[\rm(2)]
Let $n$ be a positive integer and 
$f\colon (A|_U)^{\oplus n}\to N|_U$ a morphism of ind-$A$-modules.
For any $x\in X$, there exist an open neighborhood $U$ of $x$
and a morphism $\widetilde{f|_U}\colon (A|_U)^{\oplus n}\to M|_U$ of ind-$A|_U$-modules
such that $\varphi|_U\circ\widetilde{f|_U} = f|_U$.
\end{itemize}
\end{lemma}

\begin{proof}
(1)
For any open subset $V$ of $X$,
since 
$$\Hom_{\I (A|_V)}(A|_V, M|_V)
=
\Gamma(V; \hom_A(A, M))
=
\Gamma(V; \alpha_X\ihom_A(A, M))
%\simeq
%\Gamma(V, (\alpha_XM)|_V),
$$
and there exists an isomorphism
$$\Gamma(V; \alpha_X\ihom_A(A, M))
\simeq
\Gamma(V, \alpha_XM)$$
we have an isomorphism
$$\Phi_M^V\colon \Hom_{\I (A|_V)}(A|_V, M|_V)\simeq \Gamma(V, (\alpha_XM)|_V),$$
such that
the following diagram is a commutative:
\begin{equation*}
\tag*{$(\ast)_V$}
\vcenter{
\xymatrix@C=25pt@R=25pt@M=5pt{
\Hom_{\I (A|_V)}(A|_V, M|_V) \ar@{-}[r]^-{\overset{\Phi_M^V}{\sim}}
\ar@{->}[d]_-{\Hom_{\I (A|_V)}(A|_V,\ \varphi|_V)}
& \Gamma(V, (\alpha_XM)|_V)\ar@{->}[d]^-{\Gamma(V,\ (\alpha_X\varphi)|_V)}\\
\Hom_{\I (A|_V)}(A|_V, N|_V) \ar@{-}[r]^-{\overset{\Phi_N^V}{\sim}} & \Gamma(V,(\alpha_XN)|_V).
}}\end{equation*}
%See \cite[Prop.\:5.4.11]{KS01} for the details.

Let $x\in X$.
Since the functor $\alpha_X$ is exact,
the morphism $\alpha_X\varphi\colon \alpha_XM\to \alpha_XN$ is an epimorphism
and hence the map $(\alpha_X\varphi)_x\colon (\alpha_XM)_x\to (\alpha_XN)_x$ is surjective.
Let us denote by $n_{f,x}\in (\alpha_XN)_x$ the element corresponding to
$\underset{x\in V}{\varinjlim}f|_V\in\underset{x\in V}{\varinjlim}\Hom_{\I (A|_V)}(A|_V, N|_V)$
 by the isomorphism
 $$\underset{x\in V}{\varinjlim}\Hom_{\I (A|_V)}(A|_V, N|_V)\simeq (\alpha_XN)_x.$$
Since $(\alpha_X\varphi)_x\colon (\alpha_XM)_x\to (\alpha_XN)_x$ is surjective,
there exist $m_{f,x}\in (\alpha_XM)_x$ such that $(\alpha_X\varphi)_x(m_{f,x}) = n_{f,x}$,
and hence there exist an open neighborhood $U$ of $x$
and $m_f\in \Gamma(U, (\alpha_XM)|_U)$ whose stalk at $x$ is equal to $m_{f, x}$
such that $(\alpha_X\varphi)|_U(m_f) = n_f$.
Let us set $\widetilde{f|_U}\in \Hom_{\I (A|_U)}(A|_U, M|_U)$ the element corresponding to $m_f$
 by the isomorphism
 $$\Hom_{\I (A|_U)}(A|_U, M|_U)\simeq \Gamma(U, (\alpha_XM)|_U).$$
 Then we have $\varphi|_U\circ\widetilde{f|_U} = f|_U$
 by using the commutative diagram $(\ast)_U$.
 \bigskip

(2) 
%By the induction on $n$, it is enough to show that $n=2$.
Since the proof for general $n$ is similar, we shall prove the case of $n=2$.
Let $V$ be an open subset of $X$, 
$\varphi\colon M\twoheadrightarrow N$ an epimorphism of ind-$A|_V$-modules
and $f\colon A|_V\to N$ a morphism of ind-$A|_V$-modules.
Let $x\in V$.
We denote by $i_1, i_2\colon A\to A\oplus A$ the natural embeddings,
and set $f\circ i_1 = f_1,\ f\circ i_2 = f_2$.
Since $A$ satisfies the above condition,
there exist an open neighborhood $U\subset V$ of $x$
and two morphisms $\widetilde{{f_1}|_U}\colon A|_U\to M|_U,\ \widetilde{{f_2}|_U}\colon A|_U\to M|_U$
such that $\varphi|_U\circ\widetilde{{f_1}|_U} = {f_1}|_U,\ \varphi|_U\circ\widetilde{{f_2}|_U} = {f_2}|_U$:
\[\xymatrix@C=45pt@R=45pt@M=3pt{
{} & M|_U\ar@{->>}[rr]^-{\varphi|_U} & {} & N|_U\\
{} & {} & (A|_U)^{\oplus 2}\ar@{->}[ru]_-{f|_U}\ar@{.>}[lu]^-{\widetilde{f}} & {} & {}\\
A|_U\ar@{->}[rru]_-{{i_1}|_U}\ar@{->}[rrruu]^-{{f_1}|_U}\ar@{.>}[ruu]^-{\widetilde{{f_1}|_U}}
& {} & {} & {} & 
A|_U\ar@{->}[llu]^-{{i_2}|_U}\ar@{->}[luu]_-{{f_2}|_U}\ar@{.>}[llluu]_-{\widetilde{f_2}}
}\]
Hence, we have a morphism $\widetilde{f} \colon (A|_U)^{\oplus2}\to M|_U$ such that 
$\widetilde{f}\circ {i_1}|_U = \widetilde{{f_1}|_U},\ \widetilde{f}\circ {i_2}|_U = \widetilde{{f_2}|_U}$.
Since there exist isomorphisms
\begin{align*}
(\varphi|_U\circ\widetilde{f})\circ {i_1}|_U
&\simeq 
\varphi|_U\circ(\widetilde{f}\circ {i_1}|_U)
\simeq
\varphi|_U\circ\widetilde{{f_1}|_U}
\simeq
{f_1}|_U,\\
(\varphi|_U\circ\widetilde{f})\circ {i_2}|_U
&\simeq 
\varphi|_U\circ(\widetilde{f}\circ {i_2}|_U)
\simeq
\varphi|_U\circ\widetilde{{f_2}|_U}
\simeq
{f_2}|_U
\end{align*}
we have $\varphi|_U\circ\widetilde{f} = f|_U$
by the universal property of the coproduct.
Therefore,
we have a morphism $\widetilde{f}\colon (A|_U)^{\oplus 2}\to M|_U$ of ind-$A|_U$-modules
and a commutative diagram:
\[\xymatrix@C=45pt@R=25pt{
M|_U\ar@{->>}[r]^-{\varphi|_U} & N|_U\\
(A|_U)^{\oplus 2}.\ar@{->}[u]^-{\widetilde{f}}\ar@{->}[ru]_-{f|_U} & {}
}\]
This implies that $A^{\oplus 2}$ satisfy the above condition.
\end{proof}

Let us denote by $\Icoh A$ the full subcategory of $\I A$ consisting of ind-$A$-modules
which is coherent over $A$.
Although it seems that the category $\Icoh A$ is abelian by
\cite[Prop.\:6.1.19]{KS06} and \cite[Thm.\:5.16 of \S\:3.5]{Pop},
we shall prove it in this paper.
For this purpose, let us prove some finiteness properties for ind-modules as below.

\newpage

\begin{lemma}\label{lem1}
Let us consider an exact sequence in $\I A:$
$$0\longrightarrow M'\xrightarrow{\ \mu\ } M \xrightarrow{\ \nu\ } M''\longrightarrow 0.$$ 
\begin{itemize}
\item[\rm (1)]
If $M$ is locally finitely generated over $A$, then $M''$ is also locally finitely generated over $A$.

\item[\rm (2)]
$M'$ and $M''$ are locally finitely generated over $A$, then $M$ is also locally finitely generated over $A$.
\end{itemize}
\end{lemma}

\begin{proof}
(1)
Let $x\in X$.
Since $M$ is locally finitely generated over $A$,
there exist an open neighborhood $U$ of $x$, a positive integer $n$
and an epimorphism $\varphi\colon (A|_U)^{\oplus n}\twoheadrightarrow M|_U$.
Hence we have an epimorphism 
$\nu|_U\circ \varphi \colon (A|_U)^{\oplus n}\twoheadrightarrow M|_U\twoheadrightarrow M''|_U$
of ind-$A|_U$-modules from $(A|_U)^{\oplus n}$ to $M''|_U$:
This implies that $M''$ is locally finitely generated over $A$.
%\medskip

(2)
Let $x\in X$.
Since $M'$ and $M''$ are locally finitely generated over $A$, 
there exist an open neighborhood $V$ of $x$, two positive integers $m, n$
and two epimorphisms $\varphi\colon (A|_V)^{\oplus m}\to M'|_V,\ \psi\colon (A|_V)^{\oplus n}\to M''|_V$.
Moreover by Lemma \ref{lem0} (2),
there exist an open neighborhood $U\subset V$ of $x$ and
a morphism $\widetilde{\psi}\colon (A|_U)^{\oplus} \to M|_U$ of in-$A|_U$-modules
such that $\nu|_U\circ \widetilde{\psi} = \psi|_U$.
Let us denote by $p\colon A^{\oplus m+n} \to A^{\oplus n}$ the natural epimorphism and 
by $i\colon A^{\oplus m} \to A^{\oplus m+n},\ j\colon A^{\oplus n} \to A^{\oplus m+n}$
the natural monomorphism.
Moreover, by the universal property of coproduct,
there exists a morphism $\lambda\colon (A|_U)^{\oplus m+n}\to M|_U$ of ind-$A|_U$-modules
such that $\lambda\circ j|_U = \widetilde{\psi},\
\lambda \circ i|_U= \mu|_U\circ \varphi|_U$:
\[\xymatrix@R=25pt@C=30pt@M=5pt{
(A|_U)^{\oplus m} \ar@{->}[r]^-{i|_U} \ar@{->>}[d]_-{\varphi|_U} &
(A|_U)^{\oplus (m+n)} \ar@<1.0ex>@{->}[r]^-{p|_U} \ar@{.>}[d]_-{\lambda} &
(A|_U)^{\oplus n}\ar@<0.3ex>@{^{(}->}[l]^-{j|_U}
\ar@{->}[ld]^-{\widetilde{\psi}} \ar@<-1.0ex>@{->>}[d]_-{\psi|_U}\\
M'|_U \ar@{->}[r]_-{\mu|_U} & M|_U \ar@{->}[r]_-{\nu|_U} & M''|_U.
}\]
By using the universal property of cokernels,
we have a morphism $\psi'\colon (A|_U)^{\oplus n}\to M''|_U$
such that $\psi'\circ p|_U = \nu|_U\circ \lambda$. 
Since $p|_U$ is an epimorphism and there exist isomorphisms
$$\psi|_U\circ p|_U
\simeq
\nu|_U\circ \widetilde{\psi}\circ p|_U
\simeq
\beta|_U\circ \lambda\circ j|_U \circ p|_U
\simeq
\psi'|_U\circ p|_U\circ j|_U \circ p|_U
\simeq
\psi'\circ p|_U,$$
we obtain $\psi|_U = \psi'$. 
Therefore, we have a commutative diagram whose horizontal arrows are exact:
\[\xymatrix@R=25pt@C=30pt@M=5pt{
0 \ar@{->}[r] & (A|_U)^{\oplus m} \ar@{->}[r]^-{i|_U} \ar@{->>}[d]_-{\varphi|_U}
& (A|_U)^{\oplus (m+n)} \ar@{->}[r]^-{p|_U} \ar@{->}[d]_-{\lambda}
& (A|_U)^{\oplus n} \ar@{->}[r] \ar@{->>}[d]_-{\psi|_U} & 0\\
0 \ar@{->}[r] & M'|_U \ar@{->}[r]_-{\mu|_U} & M|_U \ar@{->}[r]_-{\nu|_U} & M''|_U \ar@{->}[r] & 0.
}\]
By using the snake lemma,
we have an isomorphism $\Coker\lambda \simeq 0$ and hence $\lambda$ is an epimorphism.
This implies that $M$ is locally finitely generated over $A$.
\end{proof}

Although, it is not used in this paper,
the notion of locally finitely presented can be paraphrased as below.
\begin{sublemma}\label{sublem1}
Let $M$ be an ind-$A$-module.
The following two conditions are equivalent:
\begin{itemize}
\item[\rm(i)]
$M$ is locally finitely presented over $A$.

\item[\rm(ii)]
$M$ is locally finitely generated over $A$ and 
for any $x\in X$ there exists an open neighborhood $U$ of $x$ such that
for any positive integer $l$ and
any epimorphism $\chi\colon (A|_U)^{\oplus l}\twoheadrightarrow M|_U$ of ind-$A|_U$-modules
the kernel $\Ker\chi$ of $\chi$ is locally finitely generated over $A|_U$. 
\end{itemize}
\end{sublemma}

\begin{proof}
First, we prove that $\rm (i)\Rightarrow (ii)$.
Let us assume that $M$ is locally finitely presented over $A$.
It is clear that $M$ is locally finitely generated over $A$.
Let $x\in X$.
Since $M$ is locally finitely presented over $A$,
there exist two positive integers $m, n$, an open neighborhood $U$ of $x$
and an exact sequence
$$(A|_U)^{\oplus m}\xrightarrow{\ \varphi\ }(A|_U)^{\oplus n}\xrightarrow{\ \psi\ } M|_U\longrightarrow 0$$
 in $\I(A|_U)$ by the definition.
%Since there exists an epimorphism
%$$(A|_U)^{\oplus m}\twoheadrightarrow \Image\varphi\simto\Ker\psi,$$
%$\Ker\psi$ is locally finitely generated over $A|_U$.

Let $l$ be a positive integer and
$\chi\colon (A|_U)^{\oplus l}\twoheadrightarrow M|_U$ an epimorphism of ind-$A|_U$-modules.
By Lemma \ref{lem0} (2),
for any $z\in U$ there exist an open neighborhood $V\subset U$ of $z$
and a morphism $\widetilde{\psi|_V}\colon (A|_V)^{\oplus n}\to (A|_V)^{\oplus l}$ of ind-$A|_V$-modules
such that $\psi|_V = \chi|_V\circ \widetilde{\psi|_V}$,
and hence we have a commutative diagram:
\[\xymatrix{
0\ar@{->}[r] & \Ker\psi|_V 
\ar@{->}[r] \ar@{->}[d]_-{\mu} & (A|_V)^{\oplus n}
\ar@{->}[rd]^-{\psi|_V} \ar@{->}[r]^-{\psi|_V} \ar@{->}[d]_-{\widetilde{\psi|_V}}
& M|_V  \ar@{->}[r] \ar@{=}[d] & 0 \\
0 \ar@{->}[r] &
\Ker\chi|_V \ar@{->}[r] & (A|_V)^{\oplus l} \ar@{->}[r]_-{\chi|_V} & M|_V \ar@{->}[r] & 0.
}\]
Here, the morphism $\mu\colon \Ker\psi|_V\to \Ker\chi|_V$ of ind-$A|_V$-modules is obtained
by using the universal property of kernels.
By using the snake lemma,
we have an exact sequence
$$\Ker\mu\to \Ker\widetilde{\psi|_V}\to\Ker(\id_{M|_V})\to
\Coker\mu\to\Coker\widetilde{\psi|_V}\to\Coker(\id_{M|_V})$$
 in $\I(A|_V)$
and hence we have an isomorphism $\Coker\mu\simto \Coker\widetilde{\psi|_V}$
because $\Ker(\id_{M|_V})$ and $\Coker(\id_{M|_V})$ are $0$.
So that there exists an epimorphism $(A|_V)^{\oplus l}\twoheadrightarrow\Coker\mu$
and hence $\Coker\mu$ is locally finitely generated over $A|_V$.
Since there exist epimorphisms
$$(A|_V)^{\oplus m}\twoheadrightarrow\Image\varphi|_V\simto\Ker\psi|_V\twoheadrightarrow\Image\mu,$$
$\Image\mu$ is locally finitely generated over $A|_V$.
Since there exists an exact sequence in $\I(A|_V)$
$$0\to \Image\mu\to \Ker\chi|_V\to \Coker\mu\to 0$$
$\Ker\chi|_V$ is locally finitely generated over $A|_V$
by Lemma \ref{lem1} (2).
This implies that $\Ker\chi$ is locally finitely generated over $A|_U$.
%\medskip

We shall prove that $\rm (ii)\Rightarrow (i)$.
Let us assume that $M$ is locally finitely generated over $A$.
Let $x\in X$.
Then  there exist open neighborhood $U$ of $x$,
a positive integer $n$ and an epimorphism $\varphi\colon(A|_U)^{\oplus n}\twoheadrightarrow M|_U$
of ind-$A|_U$-modules by the definition.
By the assumption, $\Ker\varphi$ is locally finitely generated over $A|_U$,
namely, there exist an open neighborhood $V\subset U$ of $x$, a positive integer $m$
and an epimorphism $(A|_V)^{\oplus m}\twoheadrightarrow \Ker\varphi|_V$.
Then we have a commutative diagram:
\[\xymatrix{(A|_V)^{\oplus m} \ar@{->}[r] \ar@{->>}[rd]
& (A|_V)^{\oplus n} \ar@{->}[r]^-{\varphi|_V} & M|_V \ar@{->}[r] & 0\\
{} & \Ker\varphi|_V \ar@{^{(}->}[u]& {}
}\]
and hence we obtain an exact sequence in $\I(A|_V)$
$$(A|_V)^{\oplus m} \to (A|_V)^{\oplus n}\to M|_V\to 0.$$ 
This implies that $M$ is locally finitely presented over $A$.
\end{proof}

\begin{sublemma}\label{sublem2}
Let $M$ be an ind-$A$-module.

\begin{itemize}
\item[\rm (1)]
We assume that $M$ is locally finitely generated over $A$.
Thus,
for any $x\in X$ there exist an open neighborhood $U$ of $x$,
a positive integer $n$ and an epimorphism $\varphi_U\colon (A|_U)^{\oplus n}\twoheadrightarrow M|_U$.
If the kernel $\Ker\varphi_U$ of the epimorphism $\varphi_U$ is locally finitely generated over $A|_U$ for each $U$,
$M$ is locally finitely presented over $A$.

\item[\rm (2)]
We assume that $M$ is locally finitely presented over $A$.
Thus,
for any $x\in X$ there exist an open neighborhood $U$ of $x$,
positive integers $m, n$ and an exact sequence 
$$(A|_U)^{\oplus m}\xrightarrow{\ \psi_U\ }(A|_U)^{\oplus n}\xrightarrow{\ \varphi_U\ }M|_U\to0.$$
Then the kernel $\Ker\varphi_U$of $\varphi_U$ is locally finitely generated over $A|_U$ for each $U$.
\end{itemize}
\end{sublemma}

\begin{proof}
(1)
Let us assume that 
the kernel $\Ker\varphi_U$ of the epimorphism $\varphi_U$ is locally finitely generated over $A|_U$ for each $U$.
Then for any $z\in U$ there exist an open neighborhood $V\subset U$ of $z$,
a positive integer $m$ and an epimorphism $\psi\colon (A|_V)^{\oplus m}\to \Ker\varphi|_V$,
and hence we have a commutative diagram:
\[\xymatrix{(A|_V)^{\oplus m} \ar@{->}[r] \ar@{->>}[rd]_-\psi
& (A|_V)^{\oplus n} \ar@{->}[r]^-{\varphi|_V} & M|_V \ar@{->}[r] & 0\\
{} & \Ker\varphi|_V. \ar@{^{(}->}[u]& {}
}\]
Therefore, we have an exact sequence in $\I(A|_V)$
$$(A|_V)^{\oplus m} \to (A|_V)^{\oplus n}\to M|_V\to 0.$$ 
This implies that $M|_U$ is locally finitely presented over $A|_U$
and hence $M$ is locally finitely presented over $A$.
%\medskip

(2)
By the assumption,
we have a commutative diagram:
\[\xymatrix{(A|_V)^{\oplus m} \ar@{->}[r]^-{\psi} \ar@{->>}[rd]
& (A|_U)^{\oplus n} \ar@{->}[r]^-{\varphi} & M|_U \ar@{->}[r] & 0\\
{} & \Image\psi\simeq\Ker\varphi. \ar@{^{(}->}[u]& {}
}\]
This implies that the kernel $\Ker\varphi$ of $\varphi$ is locally finitely generated over $A|_U$.
\end{proof}

\begin{lemma}\label{lem2}
Let us consider an exact sequence in $\I A:$
$$0\longrightarrow M'\xrightarrow{\ \mu\ } M \xrightarrow{\ \nu\ } M''\longrightarrow 0.$$ 
\begin{itemize}
\item[\rm (1)]
If $M$ is locally finitely presented over $A$ and $M'$ is locally finitely generated over $A$,
then $M''$ is locally finitely presented over $A$.

\item[\rm (2)]
If $M$ is locally finitely generated over $A$ and $M''$ is locally finitely presented over $A$,
then $M'$ is locally finitely generated over $A$.

\item[\rm (3)]
If $M'$ and $M''$ are locally finitely presented over $A$,
then $M$ is also locally finitely presented over $A$.
\end{itemize}
\end{lemma}

\begin{proof}
(1)
Let $x\in X$.
Since $M$ is locally finitely generated over $A$,
there exist an open neighborhood $U$ of $x$, two positive integers $m, n$
and an exact sequence
$(A|_U)^{\oplus m}\longrightarrow(A|_U)^{\oplus n}\xrightarrow{\ \varphi\ } M|_U\to 0$.
By using Sublemma \ref{sublem2} (2),
the kernel $\Ker\varphi$ of $\varphi$ is locally finitely generated over $A|_U$.
Let us set $\psi := \nu|_U\circ \varphi \colon (A|_U)^{\oplus n}\twoheadrightarrow M|_U\twoheadrightarrow M''|_U$.
By using Sublemma \ref{sublem2} (1), it is enough to show that $\Ker\psi$ is locally finitely generated over $A|_U$.
Since there exists a commutative diagram whose horizontal arrows are exact:
\[\xymatrix@M=7pt{
{} & {} & 0\ar@{->}[d] & M'|_U\ar@{->}[d]^-{\mu|_U} & {}\\
{} & \Ker \varphi\ar@{->}[r]\ar@{->}[d]^-\chi & (A|_U)^{\oplus n}\ar@{->}[r]^-\varphi\ar@{=}[d]
& M|_U\ar@{->>}[d]^-{\nu|_U} \ar@{->}[r] & 0\\
0\ar@{->}[r]  & \Ker \psi\ar@{->}[r]\ar@{->}[d] & (A|_U)^{\oplus n}\ar@{->}[r]^-\psi\ar@{->}[d] 
& M''|_U\ar@{->}[d] \\
{} & \Coker\chi & 0 & 0,
}\]
where $\chi$ is given by using the universal property of kernels,
we have an isomorphism $M'|_U\simto \Coker\chi$
by using the snake lemma.
Hence we have an exact sequence in $\I(A|_U)$:
$$0\to \Ker\varphi\to \Ker\psi\to M'|_U\to0.$$
Since $M'|_U$ and $\Ker\varphi$ are locally finitely generated over $A|_U$,
 the kernel $\Ker\psi$ of $\psi$ is locally finitely generated over $A|_U$ by Lemma \ref{lem1} (2).
Therefor, $M''$ is locally finitely preseted over $A$.
%\medskip

(2)
Let $x\in X$.
Since $M$ is locally finitely generated over $A$,
there exist an open neighborhood $U$ of $x$, a positive integer $n$
and an epimorphism $\varphi\colon (A|_U)^{\oplus n}\to M|_U$.
Let us set $\psi := \nu|_U\circ \varphi \colon (A|_U)^{\oplus n}\twoheadrightarrow M|_U\twoheadrightarrow M''|_U$.
Then we have a commutative diagram whose horizontal arrows are exact:
\[\xymatrix@M=7pt{
{} & {} &  & \Ker\psi \ar@{->}[d]& {}\\
{} & 0\ar@{->}[r]\ar@{->}[d] & (A|_U)^{\oplus n}\ar@{=}[r]\ar@{->>}[d]^-\varphi
& (A|_U)^{\oplus n}\ar@{->>}[d]^-\psi \ar@{->}[r] & 0\\
0\ar@{->}[r]  & M'|_U\ar@{->}[r]^-{\mu|_U}\ar@{->}[d] & M|_U\ar@{->}[r]^-{\nu|_U}\ar@{->}[d] 
& M''|_U\ar@{->}[d] \ar@{->}[r]& 0\\
{} & M'|_U & 0 & 0.
}\]
By using the snake lemma, we have an epimorphism $\Ker\psi\twoheadrightarrow M'|_U$.
Since $M''$ is locally finitely presented over $A$,
the kernel $\Ker\psi$ of $\psi$ is locally finitely generated over $A|_U$ by Sublemma \ref{sublem2} (2),
and hence $M'|_U$ is locally finitely generated over $A|_U$
by Lemma \ref{lem1} (1).
This implies that $M'$ is locally finitely generated over $A$.
%\medskip

(3)
Let $x\in X$.
Since $M'$ and $M''$ are locally finitely generated over $A$,
there exist an open neighborhood $U$ of $x$ and two positive integers $m, n$
and two epimorphisms
$\varphi\colon (A|_U)^{\oplus n}\twoheadrightarrow M',\ \psi\colon (A|_U)^{\oplus m}\twoheadrightarrow M''$.
By using Sublemma \ref{sublem2} (2), $\Ker \varphi$ and $\Ker\psi$ are locally finitely generated over $A|_U$.
Moreover, we have a commutative diagram whose vertical arrows and 
second and third horizontal arrows are exact:
\[\xymatrix{{} & 0 \ar@{->}[d] & 0 \ar@{->}[d] & 0 \ar@{->}[d] & {}\\
0 \ar@{->}[r] & \Ker\varphi\ar@{->}[d] \ar@{.>}[r] & \Ker\chi \ar@{->}[d] \ar@{.>}[r]
& \Ker\psi \ar@{->}[d] \ar@{->}[r] & 0\\
0 \ar@{->}[r] & (A|_U)^{\oplus n}\ar@{->}[d]^-\varphi \ar@{->}[r]
& (A|_U)^{\oplus (n+m)} \ar@{->}[d]^-\chi \ar@{->}[r]
& (A|_U)^{\oplus m} \ar@{->}[d]^-\psi \ar@{->}[r] & 0 \\
0 \ar@{->}[r] & M'|_U \ar@{->}[d] \ar@{->}[r]^-{\mu|_U}
& M|_U \ar@{->}[d] \ar@{->}[r]^-{\nu|_U} & M''|_U \ar@{->}[d] \ar@{->}[r] & 0\\
{} & 0 & 0 & 0& {} 
}\]
by using the same argument in the proof of Lemma \ref{lem1} (2).
By using the nine lemma, we have an exact sequence in $\I(A|_U)$:
$$0\to \Ker\varphi\to \Ker\chi\to \Ker\psi\to0.$$
Since $\Ker \varphi$ and $\Ker\psi$ are locally finitely generated over $A|_U$,
$\Ker\chi$ is locally finitely generated over $A|_U$ by Lemma \ref{lem1} (2).
Therefore, $M|_U$ is locally finitely presented over $A|_U$ by Sublemma \ref{sublem2} (1),
and hence $M$ is locally finitely presented over $A$.
\end{proof}

The notion of pseudo-coherent can be paraphrased as below.

\begin{lemma}\label{lem3}
Let $M$ be an ind-$A$-module.
\begin{itemize}
\item[\rm (1)]
$M$ is pseudo-coherent over $A$ 
if and only if
for any open subset $U$ of $X$, any positive integer $n$
and any morphism $\varphi\colon (A|_U)^{\oplus n}\to M|_U$ of ind-$A|_U$-modules,
the kernel $\Ker\varphi$ of $\varphi$ is locally finitely generated over $A|_U$.

\item[\rm(2)]
If $M$ is coherent over $A$, $M$ is locally finitely presented over $A$.

\item[\rm (3)]
We assume that $M$ is pseudo-coherent over $A$.
An ind-$A$-submodule $N$ is coherent over $A$
if and only if
it is locally finitely generated over $A$.
\end{itemize}
\end{lemma}

\begin{proof}
(1)
First we assume that $M$ is pseudo-coherent over $A$.
Let $U$ be an open subset of $X$, $n$ a positive integer
and $\varphi\colon (A|_U)^{\oplus n}\to M|_U$ a morphism of ind-$A|_U$-modules.
Then there exists an exact sequence in $\I(A|_U)$
$$0\to \Ker\varphi\to (A|_U)^{\oplus n}\to \Image\varphi\to0.$$
Since $M$ is pseudo-coherent over $A$ and 
$\Image\varphi$ is an ind-$A|_U$-submodule of $M|_U$ which is locally finitely generated over $A|_U$,
$\Image\varphi$ is locally finitely presented over $A|_U$.
So that $\Ker\varphi$ is locally finitely generated over $A|_U$
by Lemma \ref{lem2} (2).

On the other hand, we assume that for any open subset $V$ of $X$, any positive integer $n$
and any morphism $\varphi\colon (A|_V)^{\oplus n}\to M|_V$ of ind-$A|_V$-modules,
$\Ker\varphi$ is locally finitely generated over $A|_V$.
Let $U$ be an open subset of $X$ and $N$ an ind-$A|_U$-submodule
which is locally finitely generated over $A|_U$.
Then for any $z\in U$ there exist an open neighborhood $V\subset U$ of $z$, a positive integer $n$
and an epimorphism $\psi\colon (A|_V)^{\oplus n}\to N|_V$.
Since there exist isomorphisms
$$\Ker\psi := \Ker((A|_V)^{\oplus n}\to N|_V)
\simeq \Ker((A|_V)^{\oplus n}\to N|_V\hookrightarrow M|_V)
= \Ker((A|_V)^{\oplus n}\to M|_V)$$
and $\Ker((A|_V)^{\oplus n}\to M|_V)$ is locally finitely generated over $A|_V$
by using the assumption,
$\Ker\psi$ is locally finitely generated over $A|_V$.
Therefore, $N$ is locally finitely presented over $A|_U$
by Sublemma \ref{sublem2} (1)
and hence $M$ is pseudo-coherent over $A$.
%\medskip

(2)
We assume that $M$ is coherent over $A$.
Since $M$ is locally finitely generated over $A$,
for any $x\in X$ there exist an open neighborhood $U$ of $x$,
a positive integer $n$ and an epimorphism
$\varphi\colon (A|_U)^{\oplus n}\twoheadrightarrow M|_U$
of ind-$A|_U$-modules.
Since $M$ is pseudo-coherent over $A$,
the kernel $\Ker\varphi$ of $\varphi$ is locally finitely generated over $A|_U$
by using (1).
Hence,
$M$ is locally finitely presented over $A$
by Sublemma \ref{sublem2} (1).
%\medskip

(3)
We assume that $M$ is pseudo-coherent over $A$.
Let $N$ be an ind-$A$-submodule of $M$.
It is enough to prove that if $N$ is locally finitely generated over $A$ it is pseudo-coherent over $A$.

Let $U$ be an open subset of $X$, $n$ a positive integer and 
$\varphi\colon(A|_U)^{\oplus n}\to N|_U$ a morphism of ind-$A|_U$-modules.
Note that there exist isomorphisms
\begin{align*}
\Ker\varphi := \Ker((A|_U)^{\oplus n} \to N|_U)
\simeq\Ker((A|_U)^{\oplus n} \to N|_U\hookrightarrow M|_U)
\simeq\Ker((A|_U)^{\oplus n} \to M|_U).
\end{align*}
By using (1) and the assumption that $M$ is pseudo-coherent over $A$,
$\Ker((A|_U)^{\oplus n} \to M|_U)$ is locally finitely generated over $A|_U$,
and hence $\Ker\varphi$ is locally finitely generated over $A|_U$.
Therefore, $N$ is pseudo-coherent over $A$ by using (1).
\end{proof}

Note that the property of coherent is local property by Lemma \ref{lem3} (1).
Namely, if an ind-$A$-module $M$ is coherent over $A$, 
then $A|_U$ is also coherent over $A|_U$ for any open subset $U$ of $X$.
Moreover, one can show that the category $\Icoh A$ is abelian
by using the following proposition.

\begin{proposition}\label{prop1}
Let us consider an exact sequence in $\I A$:
$$0\to M'\xrightarrow{\ \mu\ } M\xrightarrow{\ \nu\ } M''\to 0.$$ 
\begin{itemize}
\item[\rm (1)]
If $M$ is coherent over $A$ and $M''$ is locally finitely presented,
then $M'$ is coherent over $A$.

\item[\rm (2)]
$M'$ and $M''$ are coherent over $A$, then $M$ is also coherent over $A$.

\item[\rm (3)]
If $M$ is coherent over $A$ and $M'$ is locally finitely generated over $A$,
then $M''$ is coherent over $A$.
\end{itemize}
\end{proposition}

\begin{proof}
(1)
Since $M$ is locally finitely generated over $A$ and $M''$ is locally finitely presented over $A$,
$M'$ is locally finitely generated over $A$ by Lemma \ref{lem2} (2).
Let us prove that $M'$ is pseudo-coherent over $A$.

Let $U$ be an open subset of $X$, $n$ a positive integer
and morphism $\varphi\colon (A|_U)^{\oplus n} \to M'|_U$
of ind-$A|_U$-modules.
Then there exist isomorphisms
\begin{align*}
\Ker\varphi &:= \Ker((A|_U)^{\oplus n} \to M'|_U)\\
&\simeq\Ker((A|_U)^{\oplus n} \to M'|_U\hookrightarrow M|_U)\\
&\simeq\Ker((A|_U)^{\oplus n} \to M|_U).
\end{align*}
Since $M$ is pseudo-coherent over $A$,
$\Ker((A|_U)^{\oplus n} \to M|_U)$ is locally finitely generated over $A|_U$
by Lemma \ref{lem3} (1)
and hence $\Ker\varphi$ is locally finitely generated over $A|_U$.
Therefore, $M'$ is pseudo-coherent over $A$ by Lemma \ref{lem3} (1).
%\medskip

(2)
Since $M'$ and $M''$ are locally finitely generated over $A$,
$M$ is also locally finitely generated over $A$ by Lemma \ref{lem1} (2).
Let us prove that $M$ is pseudo-coherent over $A$.

Let $U$ be an open subset of $X$, $n$ a positive integer
and morphism $\varphi\colon (A|_U)^{\oplus n} \to M|_U$
of ind-$A|_U$-modules.
Let us set $\psi := \nu|_U\circ\varphi\colon (A|_U)^{\oplus n}\twoheadrightarrow M|_U\twoheadrightarrow M''|_U$.
Then we have a commutative diagram whose horizontal arrows are exact:
\[\xymatrix@M=7pt{
{} & {} & \Ker\varphi \ar@{->}[d]\ar@{->}[r] & \Ker\psi \ar@{->}[d]& {}\\
{} & 0\ar@{->}[r]\ar@{->}[d] & (A|_U)^{\oplus n}\ar@{=}[r]\ar@{->>}[d]^-\varphi
& (A|_U)^{\oplus n}\ar@{->>}[d]^-\psi \ar@{->}[r] & 0\\
0\ar@{->}[r]  & M'|_U\ar@{->}[r]^-{\mu|_U}\ar@{->}[d] & M|_U\ar@{->}[r]^-{\nu|_U}\ar@{->}[d] 
& M''|_U\ar@{->}[d] \ar@{->}[r]& 0\\
{} & M'|_U & 0 & 0.
}\]
By using the snake lemma,
we have an exact sequence in $\I(A|_U)$:
$$0\to \Ker\varphi\to\Ker\psi \to M'|_U \to 0.$$
%and hence we have an exact sequence in $\I(A|_U)$:
%$$0\to \Ker\varphi\to\Ker\psi\to \Image\lambda \to 0.$$
Since $M''$ is pseudo-coherent over $A$,
$\Ker\psi$ is locally finitely generated over $A|_U$
by Lemma \ref{lem3} (1).
%and hence $\Image\lambda$ is locally finitely generated over $A|_U$
%by Lemma \ref{lem1} (1).
Moreover, 
since $M'|_U$ is coherent over $A|_U$,
$\Ker\varphi$ is locally finitely generated over $A|_U$
by Lemma \ref{lem2} (2) and Lemma \ref{lem3} (2).
Therefore, $M$ is pseudo-coherent over $A$ by Lemma \ref{lem3} (1),
and hence $M$ is coherent over $A$.
%\medskip

(3)
By using  Lemma \ref{lem2} (1) and Lemma \ref{lem3} (2),
$M''$ is locally finitely presented over $A$.
In particular, $M''$ is locally finitely generated over $A$.
Let us prove that $M''$ is pseudo-coherent over $A$.

Let $U$ be an open subset of $X$, $n$ a positive integer
and morphism $\varphi\colon (A|_U)^{\oplus n} \to M''|_U$
of ind-$A|_U$-modules.
Since $M'$ is locally finitely generated over $A$,
there exist an open neighborhood $V\subset U$ of $x$, a positive integer $m$
and an epimorphism $\psi\colon (A|_V)^{\oplus m}\to M'|_V$.
Hence, by the same argument in the proof of Lemma \ref{lem1} (2)
we have a commutative diagram whose horizontal arrows are exact:
\[\xymatrix{0 \ar@{->}[r] & (A|_V)^{\oplus m} \ar@{->}[r] \ar@{->>}[d]^-\psi & (A|_V)^{\oplus (m+n)}
\ar@{->}[r] \ar@{.>}[d]^-{^\exists\chi} & (A|_V)^{\oplus n} \ar@{->}[r] \ar@{->}[d]^-{\varphi|_V} & 0\\
0 \ar@{->}[r] & M'|_V \ar@{->}[r] & M|_V \ar@{->}[r] & M''|_V \ar@{->}[r] & 0
}\]
and hence we have an exact sequence in $\I(A|_V)$:
$$\Ker\chi\to \Ker\varphi\to \Coker\psi\simeq0$$
 by using the snake lemma.
Since $M|_V$ is pseudo-coherent over $A|_V$,
the kernel $\Ker\chi$ of $\chi$ is locally finitely generated over $A|_V$
and hence the kernel $\Ker\varphi|_V$ of $\varphi$ is locally finitely generated over $A|_V$
by Lemma \ref{lem1} (1).
This implies that $\Ker\varphi$ is locally finitely generated over $A|_U$.
Therefore, $M''$ is pseudo-coherent over $A$
by Lemma \ref{lem3} (1).
\end{proof}

\begin{corollary}\label{cor-abel}
The category $\Icoh A$ of coherent ind-$A$-modules is abelian.
\end{corollary}

\begin{proof}
It is enough to prove that 
for any two coherent ind-$A$-modules $M, N$,
the coproduct $M\oplus N$ of $M$ and $N$ is also coherent over $A$
and 
for any morphism $\varphi\colon M\to N$ of coherent ind-$A$-modules
the kernel $\Ker\varphi$ and cokernel $\Coker\varphi$ of $\varphi$ are also coherent over $A$.
The first one follows from Proposition \ref{prop1} (2).
The second one follows from Proposition \ref{prop1} (3).
The third one follows from Proposition \ref{prop1} (1) and Lemma \ref{lem3} (2).
\end{proof}

\begin{proposition}
Let $M$ be an ind-$A$-module.
We assume that $A$ is a coherent over $A$.
Then $M$ is coherent over $A$ if and only if $A$ is locally finitely presented over $A$.
\end{proposition}

%\newpage
\begin{proof}
By Lemma \ref{lem3} (2),
it is enough to prove that if $M$ is locally finitely presented over $A$
then $M$ is coherent over $A$.

Let $M$ is locally finitely presented over $A$.
It is enough to show that $M$ is pseudo-coherent over $A$.
We shall prove it by using Lemma \ref{lem3} (1).

Let $U$ be an open subset of $X$, $n$ a positive integer
and $\varphi\colon (A|_U)^{\oplus n}\to M|_U$ a morphism of ind-$A|_U$-modules.
We shall prove that the kernel $\Ker\varphi$ of $\varphi$ is locally finitely generated over $A|_U$.
Let $x\in U$.
Since $M$ is locally finitely presented over $A$,
there exist an open neighborhood $V\subset U$ of $x$, a positive integer $m$
and an epimorphism $\psi\colon (A|_V)^{\oplus m}\twoheadrightarrow M|_V$ of ind-$A|_V$-modules
whose kernel is locally finitely generated over $A|_V$
by Sublemma \ref{sublem2} (2)
and hence we have an exact sequence in $\I(A|_V)$:
$$0\to \Ker\psi\longrightarrow (A|_V)^{\oplus m}\xrightarrow{\ \psi\ }M|_V\to 0.$$
By Lemma \ref{lem0} (2),
there exists an open neighborhood $W$ of $x$
and a morphism $\widetilde{\varphi|_W}\colon (A|_W)^{\oplus n}\to (A|_W)^{\oplus m}$
such that $\psi|_W\circ \widetilde{\varphi|_W} = \varphi|_W$.
Therefore we have a commutative diagram whose horizontal arrows are exact:
\[\xymatrix@M=7pt{
{} & 0\ar@{->}[r]\ar@{->}[d] & (A|_W)^{\oplus n}\ar@{=}[r]\ar@{->}[d]^-{\widetilde{\varphi|_W}}
& (A|_W)^{\oplus n}\ar@{->}[d]^-{\varphi|_W} \ar@{->}[r] & 0\\
0\ar@{->}[r]  & \Ker\psi|_W\ar@{->}[r] & (A|_W)^{\oplus m}\ar@{->}[r]_-{\psi|_W}
&  M|_W\ar@{->}[r]& 0,
}\]
and hence we have an exact sequence in $\I(A|_W)$:
$$0\to \Ker\widetilde{\varphi|_W}\longrightarrow\Ker\varphi|_W\xrightarrow{\ \nu\ }
\Ker\psi|_W\xrightarrow{\ \mu\ } \Coker\widetilde{\varphi|_W}.$$
by the snake lemma.
Since $A|_W$ is coherent over $A|_W$,
$\Ker\widetilde{\varphi|_W}$ and $\Coker\widetilde{\varphi|_W}$
are also coherent over $A|_W$ by Corollary \ref{cor-abel} (see also Proposition \ref{prop1}).
Since $\Ker\psi|_W$ is locally finitely generated over $A|_W$
and there exists an epimorphism $\Ker\psi|_W\twoheadrightarrow \Image\mu$
of ind-$A|_W$-modules,
$\Image\mu$ is locally finitely generated over $A|_W$
by Lemma \ref{lem1} (1).
Hence $\Image\mu$ is coherent over $A|_W$
by Lemma \ref{lem3} (3).
In particular, $\Image\mu$ is locally finitely presented over $A|_W$
by Lemma \ref{lem3} (2).
Since there exist an exact sequence in $\I(A|_W)$:
$$0\to \Ker\mu\to \Ker\psi|_W\to \Image\mu\to0$$
and $\Ker\psi|_W$ is locally finitely generated over $A|_V$,
$\Ker\mu$ is locally finitely generated over $A|_V$
by Lemma \ref{lem2} (2).
Since there exists an exact sequence in $\I(A|_W)$
$$0\to \Ker\widetilde{\varphi|_W}\to\Ker\varphi|_W\to\Ker\mu\to0,$$
$\Ker\varphi|_W$ is locally finitely generated over $A|_V$
by Lemma \ref{lem1} (2).
This implies that $\Ker\varphi$ is locally finitely generated over $A|_U$.
\end{proof}

%Recall that there exist functors:
%\begin{align*}
%\alpha_X&\colon \I(\beta_X\SA)\to \Mod(\SA)\\
%\beta_X&\colon \Mod(\SA)\to \I(\beta_X\SA)
%\end{align*}
%for a sheaf of $\Bbbk_X$-algebras $\SA$.
%\newpage

\begin{proposition}
Let $\SA$ be a sheaf of $\Bbbk_X$-algebras,
$M$ an ind-$\beta_X\SA$-modul
and $\M$ an $\SA$-module.
\begin{itemize}
\item[\rm(1)]
\begin{itemize}
\setlength{\itemsep}{2pt}
\item[\rm(i)]
If $M$ is locally finitely generated over $\beta_X\SA$,
then $\alpha_XM$ is locally finitely generated over $\SA$.

\item[\rm (ii)]
If $M$ is locally finitely presented over $\beta_X\SA$,
then $\alpha_XM$ is locally finitely presented over $\SA$.

\item[\rm(iii)]
If $M$ is pseudo-coherent over $\beta_X\SA$,
then $\alpha_XM$ is pseudo-coherent over $\SA$.

\item[\rm (iv)]
If $M$ is coherent over $\beta_X\SA$,
then $\alpha_XM$ is coherent over $\SA$.
\end{itemize}

\item[\rm (2)]
\begin{itemize}
\setlength{\itemsep}{-1pt}
\item[\rm(i)]
If $\M$ is locally finitely generated over $\SA$,
then $\beta_X\M$ is locally finitely generated over $\beta_X\SA$.

\item[\rm (ii)]
If $\M$ is locally finitely presented over $\SA$,
then $\beta_X\M$ is locally finitely presented over $\beta_X\SA$.

\item[\rm(iii)]
If $\M$ is pseudo-coherent over $\SA$,
then $\beta_X\M$ is pseudo-coherent over $\beta_X\SA$.

\item[\rm (iv)]
If $\M$ is coherent over $\SA$,
then $\beta_X\M$ is coherent over $\beta_X\SA$.
\end{itemize}
\end{itemize}
\end{proposition}

\begin{proof}
(1)(i)
We assume that $M$ is locally finitely generated over $\beta_X\SA$.
Let $x\in X$.
Then there exist an open neighborhood $U$ of $x$,
a positive integer $n$
and epimorphism $\varphi\colon ((\beta_X\SA)|_U)^{\oplus n}\twoheadrightarrow M|_U$
of ind-$(\beta_X\SA)|_U$-modules.
Since the functor $\alpha_U$ is exact
and there exist isomorphisms
$\alpha_U((\beta_X\SA)|_U) \simeq \SA|_U,\
\alpha_U(M|_U) \simeq (\alpha_XM)|_U$,
we have an epimorphism $\alpha|_U(\varphi)\colon (\SA|_U)^{\oplus n}\twoheadrightarrow (\alpha_XM)|_U$.
%This implies that 
Hence, $\alpha_XM$ is locally finitely generated over $\SA$.
%%\medskip

(1)(ii)
We assume that $M$ is locally finitely presented over $\beta_X\SA$.
Let $x\in X$.
Then there exist an open neighborhood $U$ of $x$,
positive integers $m, n$
and an exact sequence
$$((\beta_X\SA)|_U)^{\oplus m}\to ((\beta_X\SA)|_U)^{\oplus n}\to M|_U\to 0$$
in $\I((\beta_X\SA)|_U)$.
Since the functor $\alpha_U$ is exact
and there exist isomorphisms
$\alpha_U((\beta_X\SA)|_U) \simeq \SA|_U,\
\alpha_U(M|_U) \simeq (\alpha_XM)|_U$,
we have an exact sequence in $\Mod(\SA|_U)$:
$$(\SA|_U)^{\oplus m}\to (\SA|_U)^{\oplus n}\to (\alpha_XM)|_U\to 0.$$
This implies that $\alpha_XM$ is locally finitely presented over $\SA$.
%%\medskip

(1)(iii).
We assume that $M$ is pseud-coherent over $\beta_X\SA$.
Let $U$ be an open subset of $X$,
$n$ a positive integer
and $\varphi\colon (\SA|_U)^{\oplus n}\to (\alpha_XM)|_U$ a morphism of $\SA|_U$-modules.
Since the functor $\alpha_X$ is exact and there exists an isomorphism $\alpha_X\circ \beta_X\simeq \id$,
we have isomorphisms
$$\Ker\varphi\simeq \Ker((\alpha_X\beta_X\SA)|_U)^{\oplus n}\to (\alpha_XM)|_U
\simeq \alpha_U\Ker((\beta_X\SA)|_U)^{\oplus n}\to M|_U).$$
By Lemma \ref{lem3} (1),
$\Ker((\beta_X\SA)|_U)^{\oplus n}\to M|_U)$ is locally finitely generated over $\beta_X\SA$
and hence $\alpha_X\Ker((\beta_X\SA)|_U)^{\oplus n}\to M|_U)$ is locally finitely generated over $\SA$
by using (1)(i).
Therefore, $\Ker\varphi$ is locally finitely generated over $\SA|_U$.
By Lemma \ref{lem3} (1) of the case of the sheaves with ring actions which is well-known,
$\alpha_XM$ is pseudo-coherent over $\SA$.
%%\medskip

(1)(iv) follows from (1)(i) and (1)(iii).

%\newpage
%%\medskip
%\medskip
(2)(i)
We assume that $\SM$ is locally finitely generated over $\SA$.
Let $x\in X$.
Then there exist an open neighborhood $U$ of $x$,
a positive integer $n$
and epimorphism $\varphi\colon (\SA|_U)^{\oplus n}\twoheadrightarrow \M|_U$
of $\SA|_U$-modules.
Since the functor $\beta_U$ is exact
and there exist isomorphisms
$\beta_U(\SA|_U) \simeq (\beta_X\SA)|_U,\
\beta_U(\M|_U) \simeq (\beta_X\M)|_U$
we have an epimorphism $\beta_U(\varphi)\colon ((\beta_X\SA)|_U)^{\oplus n}\twoheadrightarrow (\beta_X\M)|_U$.
This implies that $\beta_XM$ is locally finitely generated over $\beta_X\SA$.
%%\medskip

(2)(ii)
We assume that $\M$ is locally finitely presented over $\SA$.
Let $x\in X$.
Then there exist an open neighborhood $U$ of $x$,
positive integers $m, n$
and an exact sequence 
$(\SA|_U)^{\oplus m}\to (\SA|_U)^{\oplus n}\to \M|_U\to 0$
in $\Mod(\SA|_U)$.
Since the functor $\beta_U$ is exact
and there exist isomorphisms
$\beta_U(\SA|_U) \simeq (\beta_X\SA)|_U,\
\beta_U(\M|_U) \simeq (\beta_X\M)|_U$,
we have an exact sequence in $\I((\beta_X\SA)|_U)$:
$$((\beta_X\SA)|_U)^{\oplus m}\to ((\beta_X\SA)|_U)^{\oplus n}\to (\beta_X\M)|_U\to 0.$$
This implies that $\beta_X\M$ is locally finitely presented over $\beta_X\SA$.
%%\medskip

(2)(iii).
We assume that $\SM$ is pseudo-coherent over $\SA$.
Let $U$ be an open subset of $X$,
$n$ a positive integer
and $\Phi\colon ((\beta_X\SA)|_U)^{\oplus n}\to (\beta_X\M)|_U$ a morphism of $(\beta_X\SA)|_U$-modules.
Since $\beta_X$ is fully faithful functor,
there exists a morphism $\varphi\colon \SA|_U\to \M|_U$ of $\SA|_U$-modules
such that $\Phi = \beta_U(\varphi)$.
By Lemma \ref{lem3} (1) of the case of the sheaves with ring actions which is well-known,
$\Ker\varphi$ is locally finitely generated over $A|_U$.
Therefore, since there exists an isomorphism $\Ker\Phi\simeq \beta_U(\Ker\varphi)$,
$\Ker\Phi$ is locally finitely generated over $(\beta_X\SA)|_U$ by (2)(i)
and hence $\beta_X\M$ is pseudo-coherent over $\beta_X\SA$ by Lemma \ref{lem3} (1).
%%\medskip

(2)(iv) follows from (2)(i) and (2)(iii).
\end{proof}

\begin{theorem}\label{main-thm}
Let $\SA$ be a sheaf of $\Bbbk_X$-algebras.
There exists an equivalence of categories:
 \[\xymatrix@C=60pt{\Modcoh(\SA)  \ar@<1.0ex>[r]^-{\beta_{X}} \ar@{}[r]|-\sim
 & \Icoh(\beta_X\SA) 
\ar@<1.0ex>[l]^-{\alpha_{X}}}.\]
\end{theorem}

\begin{proof}
It is enough to show that
the natural morphism $\eta\colon (\beta_X\circ\alpha_X)(M) \to M$ is an isomorphism
for any coherent ind-$\beta_X\SA$-module $M$.

Let $M$ be a coherent ind-$\beta_X\SA$-module.
By Lemma \ref{lem3} (2),
$M$ is locally finitely presented over $\beta_X\SA$.
Then for any $x\in X$ there exist an open neighborhood $U$ of $x$,
two positive integers $m, n$ and an exact sequence in $\I((\beta_X\SA)|_U)$:
$$((\beta_X\mathscr{A})|_U)^{\oplus m}\longrightarrow((\beta_X\mathscr{A})|_U)^{\oplus n}
\longrightarrow M|_U\longrightarrow0.$$
Let us recall that there exists an isomorphism $\alpha_X\circ \beta_X \simeq \id_{\Mod(\Bbbk_X)}$
of functors, see \cite[Prop.\:3.3.27 (iii), also p50]{KS01} for the details.
Since there exists an isomorphism
$(\beta_U\circ\alpha_U)(M|_U)\simeq((\beta_X\circ\alpha_X)(M))|_U$,
we have a commutative diagram whose horizontal arrows are exact:
\[\xymatrix{((\beta_X\mathscr{A})|_U)^{\oplus m} \ar@{->}[r] \ar@{->}[d]^-\wr\ar@{->}[r]
& ((\beta_X\mathscr{A})|_U)^{\oplus n} \ar@{->}[r] \ar@{->}[d]^-{\wr} 
&(\beta_U\circ \alpha_U)(M|_U)\ar@{->}[d]^-{\eta|_U}\ar@{->}[r]& 0 \ar@{=}[d]\ar@{->}[r]& 0\ar@{=}[d]\\
((\beta_X\mathscr{A})|_U)^{\oplus m} \ar@{->}[r] & ((\beta_X\mathscr{A})|_U)^{\oplus n} \ar@{->}[r]
&M|_U\ar@{->}[r] & 0\ar@{->}[r] & 0.
}\]
By the five lemma,
$\eta|_U$ is an isomorphism
and hence $\Ker(\eta|_U)\simeq 0,\ \Image(\eta|_U)\simeq 0$.
This implies that
for any $x\in X$ there exists an open neighborhood $U$ of $x$
such that
$$(\Ker \eta)|_U \simeq 0,\hspace{7pt}
(\Image \eta)|_U \simeq 0.$$
Here we used the fact that inverse image functors are exact.
Hence by Corollary \ref{cor-indstalk}, we have $\Ker \eta\simeq 0,\ \Image\eta\simeq 0$.
Therefore, $\eta$ is an isomorphism.  
\end{proof}

\end{document}